\documentclass[reqno,12pt]{amsart}
\usepackage{amscd,amsfonts,mathrsfs,amsthm,enumerate}
\usepackage{amssymb, amsmath}
\usepackage{stmaryrd}
\usepackage{epsfig}
\usepackage[alphabetic, sorted-cites]{amsrefs}

\thanks{The first author is supported by the grant of $\text{E}\Sigma\Pi\text{A}:$ PE1-417.}

\def\real     #1{{\mathbb R^{#1}}}

\def\dt       {\partial_{t}}

\def\equationcolor {\color{black}}
\def\textcolor     {\color{black}}

\def\bcoleq    {\begin{equation}\equationcolor}
\def\ecoleq    {\textcolor\end{equation}}
\def\bcoleqn   {\equationcolor\begin{eqnarray}}
\def\ecoleqn   {\end{eqnarray}\textcolor}

\def\gm{{\operatorname{g}_M}}
\def\gn{{\operatorname{g}_N}}

\def\gk{{\operatorname{g}_{M\times N}}}
\def\rm{{\operatorname{R}_M}}
\def\rn{{\operatorname{R}_N}}
\def\rk{{\operatorname{R}_{M\times N}}}
\def\sk{{\operatorname{s}_{M\times N}}}
\def\rind{\operatorname{R}}

\def\sind{\operatorname{s}}

\def\P{\operatorname{P}}
\def\T{\operatorname{T}}
\def\Q{\operatorname{Q}}

\def\dF{\operatorname{d}\hspace{-3pt}F}
\def\df{\operatorname{d}\hspace{-3pt}f}

\def\gind{\operatorname{g}}
\def\Gind{\operatorname{G}}

\def\sym{\operatorname{Sym}}

\DeclareMathOperator*{\Ric}{Ric}

\DeclareMathOperator*{\trace}{trace}
\DeclareMathOperator*{\rank}{rank}

\newtheorem{theorem}{Theorem}[section]
\newtheorem{mythm}{Theorem}

\newtheorem{lemma}[theorem]{Lemma}

\newtheorem{proposition}[theorem]{Proposition}

\theoremstyle{definition}
\newtheorem{remark}[theorem]{Remark}

\newcommand{\bfig}{\begin{figure}}
\newcommand{\efig}{\end{figure}}

\makeatletter
\def\pproof#1{\@ifnextchar[\opargproof
{\opargproof[\it Proof of #1.]}}
\def\opargproof[#1]{\par\noindent {\bf #1 }}

\makeatother

\numberwithin{equation}{section}

\begin{document}

\title[Mean Curvature Flow]{Homotopy of area decreasing maps by mean curvature flow}
\author[Andreas Savas-Halilaj]{\textsc{Andreas Savas-Halilaj}}
\author[Knut Smoczyk]{\textsc{Knut Smoczyk}}
\address{Andreas Savas-Halilaj\newline
Institut f\"ur Differentialgeometrie\newline
Leibniz Universit\"at Hannover\newline
Welfengarten 1\newline
30167 Hannover, Germany\newline
{\sl E-mail address:} {\bf savasha@math.uni-hannover.de}
}
\address{Knut Smoczyk\newline
Institut f\"ur Differentialgeometrie\newline
and Riemann Center for Geometry and Physics\newline
Leibniz Universit\"at Hannover\newline
Welfengarten 1\newline
30167 Hannover, Germany\newline
{\sl E-mail address:} {\bf smoczyk@math.uni-hannover.de}
}

\date{}

\begin{abstract}
Let $f:M\to N$ be a smooth area decreasing map between two Riemannian
manifolds $(M,\gm)$ and $(N,\gn)$. Under weak and natural assumptions on the
curvatures of $(M,\gm)$ and $(N, \gn)$, we prove that the mean curvature flow
provides a smooth homotopy of $f$ to a constant map.
\end{abstract}

\maketitle
%\setcounter{tocdepth}{1}
%\tableofcontents
%%%%%%%%%%%%%%%%%%%%%%%%%%%%%%%%%%%%%%%%%%%%%%%%%%%%%%%%%%%%%%%%%%%%%%%%%

\section{Introduction}
Given a continuous map $f:M\to N$ between two smooth manifolds $M$ and $N$, it is
an interesting problem to find canonical representatives in the homotopy class of
$f$. One possible approach is the \textit{harmonic map heat flow}
that was defined by Eells and Sampson in \cite{eells}. Provided that $M$ and $N$ both carry
appropriate Riemannian metrics, they proved long-time existence and convergence of the heat flow,
showing that under these  assumptions one finds harmonic representatives in a given homotopy
class. This approach is applicable usually when the target space is negatively curved. However,
in general  one can neither expect long-time existence nor convergence of the flow, in particular
for maps between spheres, since the flow usually develops singularities.

Another way to deform a smooth map $f:M\to N$ between Riemannian manifolds $(M,\gm)$ and $(N,\gn)$
is by deforming its corresponding graph
$$\Gamma(f):=\{(x,f(x))\in M\times N:x\in M\},$$
in the product space $M\times N$ via the \textit{mean curvature flow}. A graphical solution of the
mean curvature flow can be described completely in terms of a smooth family of maps $f_t:M\to N$, $t\in[0,T)$,
$f_0=f$, where $0<T\le\infty$ is the maximal time for which the smooth graphical solution exists.

In case of long-time existence of a graphical solution and convergence we would thus obtain a smooth
homotopy from $f$ to a \textit{minimal map} $f_\infty:M\to N$ as time $t$ tends to infinity. Recall
that a map is called minimal, if and only if its graph is a minimal submanifold of $M\times N$.

The first result in this direction is due to Ecker and Huisken \cite{ecker}.
They proved long-time existence of entire graphical hypersurfaces in $\real{n+1}$.
Moreover, they proved convergence to a flat subspace, if the growth rate at infinity
of the initial graph is linear. The crucial observation in their paper was that the
scalar product of the unit normal with a height vector satisfies a nice evolution equation
that can be used to bound the second fundamental form appropriately.

The complexity of the normal bundle in higher codimensions makes the situation much more complicated.
Results analogous to that of Ecker and Huisken are not available any more without further assumptions.
However, the ideas developed in the paper of Ecker and Huisken opened a new era for the study of the
mean curvature flow of submanifolds in Riemannian manifolds of arbitrary dimension and codimension (see
for example \cite{wang1}, \cite{wang2},
\cite{wang}, \cite{chen-li-tian}, \cite{smoczyk4}, \cite{smoczyk3}, \cite{tsui}, 
\cite{smoczyk2}, \cite{medos}, \cite{lee}, \cite{chau-chen-he1},
\cite{chau-chen-he2} and the references therein).

A map $f:M\to N$ is called
\textit{weakly length decreasing} if $f^{\ast}\gn\le\gm$ and \textit{strictly length decreasing},
if $f^{\ast}\gn<\gm.$ Hence a length decreasing map has the property
that its differential shortens the lengths of tangent vectors. A smooth map $f:M\to N$ is called
\textit{weakly area decreasing} if its
differential $\df$ decreases the area of two dimensional tangent planes, that is if
$$\|\df(v)\wedge\df(w)\|_{\gn}\le\|v\wedge w\|_{\gm},$$
for all $v,w\in TM$. If the differential $\df$
is strictly decreasing the area of two dimensional tangent planes, then $f$ is called \textit{strictly
area decreasing}. Analogously, we may introduce the notion of \textit{weakly} and \textit{strictly
k-volume decreasing maps}.

In \cite{wang,tsui}, Wang and Tsui studied deformations of smooth maps $f:M\to N$
between Riemannian manifolds under the mean curvature flow. Under the assumption that the initial map is
strictly area decreasing, $M$ and $N$ are compact space forms with $\dim M\ge 2$, whose corresponding
sectional curvatures $\sec_{M}$ and $\sec_{N}$ satisfy
\begin{equation*}\label{curvcond}
\sec_{M}\ge|\sec_{N}|,\quad \sec_M+\sec_N>0,
\end{equation*}
they proved long-time existence of the mean curvature flow of the graph and convergence of $f$
to a constant map. Recently, Lee and Lee \cite{lee} generalized the result of Wang
and Tsui by showing that the same result holds true provided that $M$ and $N$ are compact Riemannian
manifolds whose sectional curvatures are bounded by
\begin{equation*}\label{curvcond2}
\sec_{M}\ge\sigma\ge\sec_{N}
\end{equation*}
for some strictly positive number $\sigma>0$.

The goal of this paper is to show that the curvature assumptions can be relaxed
even much further. In particular, we show that the deformation of area decreasing
maps under its mean curvature gives the following result:

\begin{mythm}\label{thmA}
Let $M$ and $N$ be two compact Riemannian manifolds. Assume that $m=\dim M\ge 2$
and that there exists a positive constant $\sigma$ such that the sectional curvatures $\sec_{M}$
of $M$ and $\sec_{N}$ of $N$ and the Ricci curvature $\Ric_{M}$ of $M$ satisfy
\begin{equation*}\label{curvcond3}
\sec_{M}>-\sigma,\,\,\,\,{\Ric}_{M}\ge(m-1)\sigma\ge(m-1)\sec_{N}.
\end{equation*}
If $f:M\to N$ is a strictly area decreasing smooth map, then the mean curvature flow of the graph
of $f$ remains the graph of a strictly area decreasing map and exists for all time. Moreover, under
the mean curvature flow the area decreasing map converges to a constant map.
\end{mythm}

\begin{remark}
The above theorem generalizes the results in \cite{tsui} and \cite{lee} since
the curvature assumption is more general.
\end{remark}

\begin{remark}
In \cite{savas} we studied minimal graphs generated by length and area decreasing maps between
two Riemannian manifolds. From the examples presented in \cite[Subsection 3.6]{savas}, it follows that
the imposed curvature assumptions in Theorem \ref{thmA} are optimal.
\end{remark}

\begin{remark}
According to Theorem \ref{thmA}, any strictly area decreasing map between two compact Riemannian
manifolds $(M,\gm)$ and $(N,\gn)$ satisfying these curvature assumptions must be null-homotopic.
Such a result fails to hold for $3$-volume decreasing maps since Guth \cite{guth} showed that
there are infinitely many non null-homotopic $3$-volume decreasing maps between unit euclidean
spheres.
\end{remark}

At this point let us say some words about the proof of Theorem \ref{thmA}.
Since the manifold $M$ is assumed to be compact, short time existence
of the mean curvature flow is guaranteed.
Moreover, it follows by continuity that there is an interval where the solution remains a graph.
The first step in our proof is to show that the assumption of being area decreasing is
preserved by the mean curvature flow. As a consequence of this fact, it follows that the norm of
the differential of the initial map remains bounded in time. This fact implies that the
deformation of the graph via the mean curvature flow remains a graph as long as the solution exists.
The second step is to prove that the flow exists for long time. In general, this can be achieved by
showing that the norm of the second fundamental form remains bounded in time. However, such a bound
is not available.
Following ideas developed by Wang and by Tsui and Wang in \cite{wang}, \cite{tsui} we introduce an
angle-type function on $M$. We show then that under our curvature assumptions this function satisfies a
nice differential inequality involving also the squared norm of the second fundamental form.
The idea now is to compare the norm of the second fundamental form with this angle-type function.
Following the same strategy as in \cite{wang} one can verify that there are no finite time singularities
of the flow. At this point a deep regularity theorem of White \cite{white} is needed. Hence in this
way it is shown that the flow exists for all time. Going back to the evolution equation of the
special angle-type function we conclude that under our assumptions the solution converges
smoothly to a constant map at infinity.

The organization of the paper is as follows.
In Section \ref{sec 2} we recall
some basic facts from the geometry of graphs. In Section \ref{sec 3} we provide the evolution equations
and the basic estimates which are used in the proof of our result. In Section \ref{sec 3}
we give the proof of Theorem \ref{thmA}.

\section{Geometry of graphs}\label{sec 2}
The purpose of this section is to set up the notation and to give some basic definitions. We shall
follow closely the notations in \cite{savas}.
\subsection{Basic facts}
Let $(M,\gm)$ and $(N,\gn)$ be Riemannian manifolds of dimension $m$ and $n$, respectively.
The induced metric on the ambient space $M\times N$ will be denoted by $\gk$ or by
$\langle\cdot\,,\cdot\rangle$, that is
$$\gk=\langle\cdot\,,\cdot\rangle:=\gm\times \gn.$$
The \textit{graph} of a map $f:M\to N$ is defined to be the submanifold
$$\Gamma(f):=\{(x,f(x))\in M\times N:x\in M\}$$
of $M\times N$. The graph $\Gamma(f)$ can be parametrized via the embedding $F:M\to M\times N$,
$F:=I_{M}\times f$, where $I_{M}$ is the identity map of $M$.
The Riemannian metric induced by $F$ on $M$ will be denoted by
$$\gind:=F^*\gk.$$
The two natural projections $\pi_{M}:M\times N\to M$ and $\pi_{N}:M\times N\to N$
are submersions, that is they are smooth and have maximal rank. The tangent bundle
of the product manifold $M\times N$, splits as a direct sum
\begin{equation*}
T(M\times N)=TM\oplus TN.
\end{equation*}
The four metrics $\gm,\gn,\gk$ and $\gind$ are related by
\begin{eqnarray}
\gk&=&\pi_M^*\gm+\pi_N^*\gn\,,\label{met1}\\
\gind&=&F^*\gk=\gm+f^*\gn\,.\label{met2}
\end{eqnarray}
Additionally, we define the symmetric $2$-tensors
\begin{eqnarray}
\sk&:=&\pi_M^*\gm-\pi_{N}^*\gn\,,\label{met3}\\
\sind&:=&F^*\sk=\gm-f^*\gn\,.\label{met4}
\end{eqnarray}
The Levi-Civita connection $\nabla^{\gk}$ associated to the Riemannian metric $\gk$ on $M \times N$ is
related to the Levi-Civita connections $\nabla^{\gm}$ on $(M,\gm)$ and $\nabla^{\gn}$ on
$(N,\gn)$ by
$$\nabla^{\gk}=\pi_M^*\nabla^{\gm}\oplus\pi_N^*\nabla^{\gn}\,.$$
The corresponding curvature tensor $\rk$ on $M\times N$ with respect to the metric
$\gk$ is related to the curvature
tensors $\rm$ on $(M,\gm)$ and $\rn$ on $(N,\gn)$ by
\begin{equation*}
\rk=\pi^{*}_{M}\rm\oplus\pi^{*}_{N}\rn.
\end{equation*}
The Levi-Civita connection on $M$ with respect to the induced metric
$\gind=F^*\gk$ is denoted by $\nabla$, the curvature tensor by $\rind$ and the Ricci curvature by $\Ric$.

\subsection{The second fundamental form}
Let $F^{\ast}TN$ denote the tangent bundle of $N$ along $M$. Note that by definition the fibers of $F^*TN$
at $x\in M$ coincide with $T_{F(x)}N$.
The differential $\dF$ of $F$ is then a section in the bundle $F^*TN\otimes T^*M$.
In the sequel we will denote all full connections on bundles over $M$ that are induced by
the Levi-Civita connection on $N$ via the immersion $F:M\to N$ by the same letter $\nabla$.
The covariant derivative of
$\dF$ is called the \textit{second fundamental form} of the immersion $F$
and it will be denoted by $A$. That is
$$A(v,w):=(\nabla\hspace{-2pt}\dF)(v,w)=\nabla^{\gk}_{\dF(v)}{\dF(w)}-\dF(\nabla_{v}w),$$
for any vector fields $v,w\in TM$.
The second fundamental form $A$ is a symmetric tensor and takes values in
the normal bundle $\mathcal{N}M$ of the graph $\Gamma(f)$. Since $\mathcal{N}M$ is a subbundle
of $F^*TN$, the full connection $\nabla$ can be used on $\mathcal{N}M$. By projecting to the
normal bundle again, we obtain the connection on the normal bundle $\mathcal{N}M$ of the graph,
which will be denoted by the symbol $\nabla^{\perp}$. If $\xi$ is a normal vector of the
graph, then the symmetric tensor $A_{\xi}$ given by
$$A_{\xi}(v,w):=\langle A(v,w),\xi\rangle$$
is called the \textit{second fundamental form with respect to the direction $\xi$}.

The trace of $A$ with respect to the metric $\gind$ is called the \textit{mean curvature
vector field} of $\Gamma(f)$ and it will be denoted by
$$H:={\trace}_{\gind}A.$$
Note that the mean curvature vector field $H$ is a section of the normal bundle of $\Gamma(f)$.
In the case where $H$ vanishes identically, the graph $\Gamma(f)$ is called \textit{minimal}.

By \textit{Gau\ss' equation} the curvature tensors $\rind$ and $\rk$
are related by the formula
\begin{eqnarray}
\rind(v_1,w_1,v_2,w_2)&=&(F^*\rk)(v_1,w_1,v_2,w_2)\nonumber\\
&&+\gk\bigl( A(v_1,v_2),A(w_1,w_2)\bigr)\nonumber\\
&&-\gk\bigl( A(v_1,w_2),A(w_1,v_2)\bigr),\label{gauss}
\end{eqnarray}
for any $v_1,v_2,w_1,w_2\in TM$. Moreover, the second fundamental form satisfies the
\textit{Codazzi equation}
\begin{eqnarray}
(\nabla_uA)(v,w)-(\nabla_vA)(u,w)&=&\rk\bigl(\dF(u),\dF(v)\bigr)\dF(w)\nonumber\\
&&-\dF\bigl(\rind(u,v)w\bigr),\label{codazzi}
\end{eqnarray}
for any $u,v,w$ on $TM$.

\subsection{Singular decomposition}
As in \cite{savas}, for any fixed point $x\in M$, let
$$\lambda^2_{1}(x)\le\cdots\le\lambda^2_{m}(x)$$
be the eigenvalues of $f^{*}\gn$ with respect to $\gm$. The corresponding values $\lambda_i\ge 0$,
$i\in\{1,\dots,m\}$, are usually called
\textit{singular values} of the differential $\df$ of $f$ and give rise to continuous functions on $M$. Let
$$r=r(x)=\rank\df(x).$$
Obviously, $r\le\min\{m,n\}$ and $\lambda_{1}(x)=\cdots=\lambda_{m-r}(x)=0.$
It is well known that the singular values can be used to define the so called \textit{singular decomposition} of $\df$.
At the point $x$ consider an orthonormal basis
$$\{\alpha_{1},\dots,\alpha_{m-r};\alpha_{m-r+1},
\dots,\alpha_{m}\}$$
with respect to $\gm$ which diagonalizes $f^*\gn$. Moreover, at the point
$f(x)$ consider an orthonormal basis
$$\{\beta_{1},\dots,\beta_{n-r};\beta_{n-r+1},\dots,\beta_{n}\}$$
with respect to $\gn$ such that
$$\df(\alpha_{i})=\lambda_{i}(x)\beta_{n-m+i},$$
for any $i\in\{m-r+1,\dots,m\}$.

Then one may define a special basis for the tangent and the normal space of the graph
in terms of the singular values. The vectors
\begin{equation}
e_{i}:=\left\{
\begin{array}{ll}
\alpha _{i} & , 1\le i\le m-r,\\
&  \\
\frac{1}{\sqrt{1+\lambda _{i}^{2}\left( x\right) }}\left( \alpha
_{i}\oplus \lambda _{i}\left(x\right) \beta _{n-m+i}\right)  &, m-r+1\leq
i\leq m,
\end{array}
\right.\label{tangent}
\end{equation}
form an orthonormal basis with respect to the metric $\gk$ of the tangent space
$\dF\hspace{-2pt}\left(T_{x}M\right)$ of the graph $\Gamma(f)$ at
$x$. Moreover, the vectors
\begin{equation}
\xi_{i}:=\left\{
\begin{array}{ll}
\beta _{i} & , 1\leq i\leq n-r,\\
&  \\
\frac{1}{\sqrt{1+\lambda _{i+m-n}^{2}\left( x\right) }}\left( -\lambda
_{i+m-n}(x)\alpha _{i+m-n}\oplus \beta _{i}\right) & , n-r+1\leq i\leq n, \\
\end{array}
\right.\label{normal}
\end{equation}
give an orthonormal basis with respect to  $\gk$ of the normal space $\mathcal{N}_{x}M$ of the
graph $\Gamma(f)$ at the point $F(x)$. From the formulas above, we deduce that
\begin{equation}
\sk(e_{i},e_{j})=\frac{1-\lambda^{2}_{i}}{1+\lambda^{2}_{i}}\delta_{ij},\quad 1\le i,j\le m.
\end{equation}
Consequently, the eigenvalues $\mu_1,\mu_2,\dots,\mu_m$ of the symmetric $2$-tensor $\sind$ with
respect to $\gind$, are
$$\mu_1:=\frac{1-\lambda^{2}_{m}}{1+\lambda^{2}_{m}}
\le\dots\le\mu_m:=\frac{1-\lambda^{2}_{1}}{1+\lambda^{2}_{1}}.$$
As it was observed in \cite{tsui, lee}, for any pair of indices $i,j$ we have
$$\mu_i+\mu_j=\frac{2(1-\lambda_i^2\lambda_j^2)}{(1+\lambda_i^2)(1+\lambda_j^2)}.$$
Hence, the graph is strictly area decreasing, if and only if the tensor $\sind$ is strictly
$2$-positive.

Moreover we get,
\begin{eqnarray}
\hspace{-.5cm}
\sk(\xi_{i},\xi_{j})
&=&\begin{cases}
\displaystyle
-\delta_{ij}&,\,1\le i\le n-r,\\[4pt]\displaystyle
-\frac{1-\lambda^{2}_{i+m-n}}{1+\lambda^{2}_{i+m-n}}\delta_{ij}&,\, n-r+1\le i\le n.
\end{cases}\label{normal2}
\end{eqnarray}
and
\begin{equation}
\sk(e_{m-r+i},\xi_{n-r+j})=-\frac{2\lambda_{m-r+i}}{1+\lambda^{2}_{m-r+i}}\delta_{ij},
\quad 1\le i,j\le r.\label{mixed}
\end{equation}

\subsection{Area decreasing maps}
For any smooth map $f:M\to N$ its differential $\df$ induces the natural map
$\Lambda^{2}\df:\Lambda^{2}TM\to\Lambda^{2}TM$,
$$\Lambda^{2}\df(v,w):=\df(v)\wedge\df(w),$$
for any $v,w\in TM.$ The map $\Lambda^{2}\df$ is called the $2$-\textit{Jacobian} of $f$.
The \textit{supremum norm} of the $2$-Jacobian is defined as the supremum of
$$\sqrt{f^\ast\gn(v_i,v_i)f^\ast\gn(v_j,v_j)-f^\ast\gn(v_i,v_j)^2},$$
where $\{v_1,\dots,v_m\}$ runs over all orthonormal bases of $TM$. A smooth map $f:M\to N$
is called \textit{weakly area decreasing} if $\|\Lambda^{2}\df\|\le 1$ and \textit{strictly area decreasing} if
$\|\Lambda^{2}\df\|<1.$
The above notions are expressed in terms of the singular values by the inequalities
$$\lambda_{i}^2\lambda_{j}^2\le1\quad\text{and}\quad\lambda_{i}^2\lambda_{j}^2<1,$$
for any $1\le i<j\le m$, respectively. On the other hand, as already noted in the previous section,
the sum of two eigenvalues of the tensor $\sind$ with respect to $\gind$ equals
$$\frac{1-\lambda^{2}_{i}}{1+\lambda^{2}_{i}}+\frac{1-\lambda^{2}_{j}}{1+\lambda^{2}_{j}}=
\frac{2(1-\lambda^{2}_{i}\lambda^{2}_{j})}{(1+\lambda^{2}_{i})(1+\lambda^{2}_{j})}.$$
Hence, the strictly area-decreasing property of $f$ is equivalent to the $2$-\textit{positivity}
of $\sind$.

The $2$-positivity of a tensor $\T\in\sym(T^*M\otimes T^*M)$ can be expressed as the
positivity of another tensor $\T^{[2]}\in\sym(\Lambda^{2}T^*M\otimes\Lambda^{2}T^*M)$.
Indeed, let $\P$ and $\Q$ be two symmetric $2$-tensors. Then, the Kulkarni-Nomizu product
$\P \odot\Q$ given by
\begin{eqnarray*}
(P\odot\Q)(v_1\wedge w_1,v_2\wedge w_2)&=&\P(v_1,v_2)\Q(w_1,w_2)+\P(w_1,w_2)\Q(v_1,v_2) \\
&-&\P(w_1,v_2)\Q(v_1,w_2)-P(v_1,w_2)\Q(w_1,v_2)
\end{eqnarray*}
is an element of $\sym(\Lambda^{2}T^*M\otimes\Lambda^{2}T^*M)$.
Now, to every element $\T\in\sym(T^*M\otimes T^*M)$ let us assign an element $\T^{[2]}$ of the bundle
$\sym(\Lambda^{2}T^*M\otimes\Lambda^{2}T^*M)$, by setting
$$\T^{[2]}:=\T\odot\gind.$$
We point out that the Riemannian metric $\Gind$ of $\Lambda^{2}TM$ is given by
$$\Gind=\tfrac{1}{2}\gind\odot\gind=\tfrac{1}{2}\gind^{[2]}.$$
The relation between the eigenvalues of $\T$ and the eigenvalues of
$\T^{[2]}$ is explained in the following lemma:
\begin{lemma}
Suppose that $\T$ is a symmetric $2$-tensor with eigenvalues $\mu_{1}\le\cdots\le\mu_{m}$
and corresponding eigenvectors $\{v_1,\dots,v_{m}\}$  with
respect to
$\gind$. Then the eigenvalues of the symmetric $2$-tensor $\T^{[2]}$ with respect to $\Gind$
are
$$\mu_{i}+\mu_{j},\quad 1\le i<j\le m,$$
with corresponding eigenvectors
$$v_{i}\wedge v_{j},\quad 1\le i<j\le m.$$
\end{lemma}

\section{Evolution of Graphs Under The Mean Curvature Flow}\label{sec 3}
In the present section we shall derive the evolution equations of some important quantities.
We mainly follow the setup and presentation used in \cite{smoczyk1,savas}.
\subsection{Mean curvature flow}
Let $M$ and $N$ be Riemannian manifolds, $f_0:M\to N$ a smooth map and
$F_{0}:=(I_{M},f_0):M\to M\times N$. Then, by a classical result, there exists a maximal
positive time $T$ for which a smooth solution $F:M\times[0,T)\to M\times N$ of the mean
curvature flow
$$\frac{dF}{dt}(x,t)=H(x,t)$$
with initial condition
$$F(x,0):=F_0(x)$$
exists. Here $H(x,t)$ denotes the mean curvature vector field at the point $x\in M$ of the
immersion $F_t:M\to M\times N$, given by
$$F_t(x):=F(x,t).$$
In this case we say that the graph $\Gamma(f_0)$ evolves by \textit{mean curvature flow} in $M\times N$.

Let $\Omega_M$ be the volume form on the Riemannian manifold $(M,\gm)$ and extend it to a parallel
$m$-form on the product manifold $M\times N$ by pulling it back via the natural projection
$\pi_M:M\times N\to M$, that is consider the $m$-form $\pi_M^*\Omega_M$. Define now the time
dependent smooth function $u:M\times[0,T)\to\real{}$, given by
$$u:=\ast\Omega_{t},$$
where here $\ast$ is the Hodge star operator with respect to the induced Riemannian metric $\gind$ and
$$\Omega_{t}:=F_{t}^*(\pi_M^*\Omega_M)=(\pi_M\circ F_{t})^*\Omega_M.$$
Note that the function $u$ is the Jacobian of the projection map from $F_{t}(M)$ to $M$. From the
implicit map theorem it follows that $u>0$ if and only if there exists a diffeomorphism $\phi_{t}:M\to M$
and a map $f_{t}:M\to N$ such that
$$F_{t}\circ\phi_{t}=(I_M,f_{t}).$$
In other words the function $u$ is positive if and only if the solution of the
mean curvature flow remains a graph. From the compactness of $M$ and the continuity of $u$, it follows that
$F_{t}$ will stay a graph at least in an interval $[0,T_{g})$ with $T_{g}\le T$. In general $T_g$ can be
strictly less than $T$. However, as we shall see in the sequel, under our curvature assumptions $T_g=T$.

\subsection{Evolution equations}
In this subsection we shall compute the evolution equations and estimate various geometric quantities
that we will need in the proof of our main result. In order to control the smallest eigenvalue of $\sind$, let us define
the symmetric $2$-tensor
$$\Phi:=\sind-\frac{1-c}{1+c}\gind,$$
where $c$ is a time-dependent function.

The evolution of the symmetric $2$-tensor $\Phi$ under the mean curvature flow is given in the following lemma.

\begin{lemma}
The evolution equation of the tensor $\Phi$ for $t\in[0,T_g)$ is given by the following formula:
\begin{eqnarray*}
\bigl(\nabla\hspace{-2pt}_{\dt}\hspace{-2pt}\Phi-\Delta\Phi\bigr)(v,w)&=&-\Phi(\operatorname{Ric}v,w)-
\Phi(\operatorname{Ric}w,v) \\
&-&2\sum_{k=1}^m(\sk-\frac{1-c}{1+c}\gk)(A(e_k,v),A(e_k,w)) \\
&-&\frac{4}{1+c}\sum_{k=1}^m\big(f_{t}^*\rn-c\rm\big)(e_k,v,e_k,w)\\
&+&\frac{c\,^{\prime}}{(1+c)^2}\gind,
\end{eqnarray*}
where $\{e_1,\dots,e_m\}$ is any orthonormal frame with respect to $\gind$.
\end{lemma}

\begin{proof}
The proof is straightforward and similar to \cite[Lemma 3.2]{savas}.
What we just need to take into account here, is that
$$\bigl(\nabla\hspace{-2pt}_{\dt}\hspace{-2pt}\gind\bigr)(v,w)=-2\gk\big(H,A(v,w)\big)$$
and
$$\bigl(\nabla\hspace{-2pt}_{\dt}\hspace{-2pt}\sind\bigr)(v,w)=\sk\big(\nabla_vH,\dF(w)\big)+\sk\big(\nabla_wH,\dF(v)\big).$$
This completes the proof.
\end{proof}

\begin{lemma}
Under the assumptions made in Theorem \ref{thmA}, the strictly area decreasing property is preserved under the
mean curvature flow for any time $t\in[0,T_g).$
\end{lemma}

\begin{proof}
The proof follows steps in our previous paper \cite[Subsection 3.5, Proof of Theorem D]{savas}.
For the sake of completeness let us briefly describe the idea of the proof. Since the initial map is strictly area decreasing, there exists a positive number $\rho_0$ such that
$$\sind^{[2]}-\rho_0\Gind\ge 0.$$
Here $\rho_0$ is just the minimum of the smallest eigenvalue of $\sind^{[2]}$ at time $0$. We claim now that the
above inequality is preserved under the mean curvature flow. In order to show this, let us
introduce the symmetric $2$-tensor
$$M_{\varepsilon}:=\sind^{[2]}-\rho_0\Gind+\varepsilon\,t\Gind,$$
where $\varepsilon$ is a positive number. Consider any $T_1<T_g$. It suffices to show that $M_{\varepsilon}>0$ on
the interval $[0,T_1]$ for all $\varepsilon<\rho_0/T_1.$ Assume in contrary that this is not true. Then, there will be a
first time $t_0\in(0,T_1)$ such that $M_{\varepsilon}$ is non-negatively definite in $[0,t_0]$ and there is a null-eigenvector
$v$ for $M_{\varepsilon}$ at some point $(x_0,t_0)$. Then, according to the second derivative criterion (see \cite[Theorem 9.1]{hamilton2}), at this point it holds
$$M_{\varepsilon}(v,v)=0,\,\,\big(\nabla M_{\varepsilon}\big)(v,v)=0\,\,\,\text{and}\,\,\,
\big(\nabla_{\dt}M_{\varepsilon}-\Delta M_{\varepsilon}\big)(v,v)\le 0.$$
From the first condition we get that at $(x_0,t_0)$ it holds
$$\lambda^2_m\lambda^2_{m-1}<1\quad\text{and}\quad \lambda^2_{i}<1,$$
for any $i\in\{1,\dots,m-1\}$. Carrying out the same estimates as in \cite[Subsection 3.5, Proof of Theorem D]{savas},
we get that at $(x_0,t_0)$ it holds
$$\big(\nabla_{\dt}M_{\varepsilon}-\Delta M_{\varepsilon}\big)(v,v)\ge\varepsilon>0$$
which is a contradiction. Thus the strictly area decreasing property is preserved under the mean curvature flow.
This completes the proof of our claim.
\end{proof}

\begin{proposition}
Under the assumptions of Theorem \ref{thmA}, the solution of the mean curvature flow
remains a graph as long as the flow exists.
\end{proposition}
\begin{proof}
As we mentioned before, there exists a time $T_g$ such that $F_t$ is graphical for any $t$ in the
interval $[0,T_g)$. We claim that $T_g=T$. Arguing indirectly, let us assume that $T_g<T$ and that
$F_{T_g}$ is not graphical. Since $\sind$ stays positive in time, there exists a positive universal
constant $\varepsilon$ such that
$$\frac{1-\lambda^2_{i}\lambda^2_{j}}{(1+\lambda^2_{i})(1+\lambda^{2}_{j})}\ge\varepsilon>0,$$
for any $1\le i<j\le m.$ In particular, we have
$$\varepsilon(1+\lambda^2_i)\le\frac{1-\lambda^2_i\lambda^2_j}{1+\lambda^2_j}\le 1,$$
for any $1\le i<j\le m.$  From the above inequality we can see that
$$1+\lambda^{2}_{i}\le\frac{1}{\varepsilon},$$
for every index $1\le i\le m.$ Hence under the curvature conditions of Theorem \ref{thmA}, the
singular values of $\df_{t}$ are bounded by a time-independent universal constant for every
$t\in[0,T_g)$. For any fixed arbitrary $x\in M$, the continuity of $u$ implies that
\begin{eqnarray*}
u(x,T_g)&=&\lim_{t\nearrow T_g}u(x,t) \\
&=&\lim_{t\nearrow T_g}\frac{1}{\sqrt{(1+\lambda^2_{1}(x,t))\cdots(1+\lambda^2_{m}(x,t))}}\\
&>&0.
\end{eqnarray*}
Thus,
$$u(x,T_g)>0$$
for any $x\in M$. This implies that the map $F_{T_g}$ must be graphical,
contradicting our initial assumption on $T_g$. Therefore $T_g=T$. This completes the proof.
\end{proof}

In the next lemma we give the evolution equation of the function $u$. The proof can be found in \cite{wang}
and for this reason will be omitted.

\begin{lemma}\label{lem evol}
The function $\log u$ evolves under the mean curvature flow for $t\in[0,T)$ according to
\begin{eqnarray*}
\nabla\hspace{-2pt}_{\dt}\hspace{-2pt}\log u&-&\Delta\log u=\|A\|^2 +\sum^{m}_{k=1}\sum^{r}_{i=1}\lambda_{m-r+i}^{2}A_{\xi_{n-r+i}}^2(e_i,e_k)\\
&+&2\sum_{1\le k\le m}\sum_{1\le i<j\le r}\lambda_{m-r+i}\lambda_{m-r+j}A_{\xi_{n-r+j}}(e_i,e_k)A_{\xi_{n-r+i}}(e_j,e_k)\\
&+&\sum_{l,k=1}^{m}\bigr(\lambda^2_{l}\rm-f^{\ast}\rn\bigl)(e_l,e_k,e_l,e_k),
\end{eqnarray*}
where here $\{e_1,\dots,e_m\}$ and  $\{\xi_1,\dots,\xi_n\}$ are the special bases defined in Subsection $2.3$.
\end{lemma}

\begin{lemma}\label{lem A}
Under the assumptions made in Theorem \ref{thmA}, for any $t\in[0,T)$ there exists
a positive number $\delta$ such that
\begin{eqnarray*}
\mathcal{A}:&=&\|A\|^2 +\sum^{m}_{k=1}\sum^{r}_{i=1}\lambda_{m-r+i}^{2}A_{\xi_{n-r+i}}^2(e_i,e_k)\\
&&+2\sum_{1\le k\le m}\sum_{1\le i<j\le r}\lambda_{m-r+i}\lambda_{m-r+j}A_{\xi_{n-r+j}}(e_i,e_k)A_{\xi_{n-r+i}}(e_j,e_k)\\
&\ge&\delta\|A\|^2.
\end{eqnarray*}
\end{lemma}

\begin{proof}
Because the strictly area decreasing property is preserved under the mean curvature flow, there exists a positive number
$\delta$ such that
$$\lambda_{i}\lambda_{j}\le 1-\delta,$$
for any $t\in [0,T)$ and $1\le i<j\le r$. Thus, for any $1\le k\le m$, we obtain that
\begin{eqnarray*}
&&\sum_{1\le i<j\le r}|\lambda_{m-r+i}\lambda_{n-r+j}A_{\xi_{n-r+j}}(e_i,e_k)A_{\xi_{n-r+i}}(e_j,e_k)|\\
&&\quad\quad\quad\quad\le (1-\delta)\sum_{1\le i<j\le r}|A_{\xi_{n-r+j}}(e_i,e_k)A_{\xi_{n-r+i}}(e_j,e_k)|.
\end{eqnarray*}
Therefore,
\begin{eqnarray*}
\mathcal{A}&\ge&\delta\|A\|^2+(1-\delta)\|A\|^2 \\
&&-2(1-\delta)\sum_{1\le k\le m}\sum_{1\le i<j\le r}|A_{\xi_{n-r+j}}(e_i,e_k)A_{\xi_{n-r+i}}(e_j,e_k)| \\
&\ge&\delta\|A\|^2+(1-\delta)\sum^{m}_{k=1}\sum_{i,j=1}^{r}A^2_{\xi_{n-r+i}}(e_j,e_k) \\
&&-2(1-\delta)\sum_{1\le k\le m}\sum_{1\le i<j\le r}|A_{\xi_{n-r+j}}(e_i,e_k)A_{\xi_{n-r+i}}(e_j,e_k)| \\
&\ge&\delta\|A\|^2\\
&&+(1-\delta)\sum_{1\le k\le m}\sum_{1\le i<j\le r}\Bigl(|A_{\xi_{n-r+j}}(e_i,e_k)|-|A_{\xi_{n-r+i}}(e_j,e_k)|\Bigr)^2\\
&\ge&\delta\|A\|^2.
\end{eqnarray*}
This completes the proof.
\end{proof}

The next estimate will be crucial in the proof of Theorem \ref{thmA}. This estimate exploits a decomposition
formula for the curvature components that we obtained in \cite{savas}.
It makes it possible to relax the
curvature assumptions used in the paper by Lee and Lee \cite{lee} to those stated
in our main theorem.
\begin{lemma}\label{lem B}
Under the assumptions made in Theorem \ref{thmA} we have
$$\mathcal{B}:=\sum_{l,k=1}^{m}\bigr(\lambda^2_{l}\rm-f^{\ast}\rn\bigl)(e_l,e_k,e_l,e_k)\ge 0.$$
\end{lemma}

\begin{proof}
In terms of the singular values, we get
$$\sind(e_k,e_k)=\gm(e_k,e_k)-f^*\gn(e_k,e_k)=\frac{1-\lambda_k^2}{1+\lambda_k^2}.$$
Since
$$1=\gind(e_k,e_k)=\gm(e_k,e_k)+f^*\gn(e_k,e_k)$$
we derive
$$\gm(e_k,e_k)=\frac{1}{1+\lambda_k^2},\quad f^*\gn(e_k,e_k)=\frac{\lambda_k^2}{1+\lambda_k^2}$$
and
$$2\gm(e_k,e_k)=\frac{1-\lambda_k^2}{1+\lambda_k^2}+1,\quad -2f^*\gn(e_k,e_k)=\frac{1-\lambda_k^2}{1+\lambda_k^2}-1.$$
Note also that for any $k \neq l$ we have
$$\gm(e_k,e_l)=f^*\gn(e_k,e_l)=\gind(e_k,e_l)=0.$$
We compute
\begin{eqnarray*}
\mathcal{B}(l):&=&2\sum_{k=1}^m\Bigl(\lambda^2_l\rm(e_k,e_l,e_k,e_l)-f^*\rn(e_k,e_l,e_k,e_l)\Bigr) \\
&=&2\sum^m_{k=1,\,k\neq l}\lambda^2_l\sec_M(e_k\wedge e_l)\gm(e_k,e_k)\gm(e_l,e_l)\\
&&-2\sum^m_{k=1,\,k\neq l}\sigma_N(\df(e_k)\wedge \df(e_l))f^*\gn(e_k,e_k)f^*\gn(e_l,e_l).
\end{eqnarray*}
Here the terms $\sigma_N(\df(e_k)\wedge \df(e_l))$ are zero if $\df(e_k)$, $\df(e_l)$ are collinear and otherwise they denote the
sectional curvatures on $(N,\gn)$ of the planes $\df(e_k)\wedge\df (e_l)$. Now the formula for $\gm(e_k,e_k)$ implies
\begin{eqnarray*}
\mathcal{B}(l)&=&\sum^m_{k=1,\,k\neq l}\left(1+\frac{1-\lambda_k^2}{1+\lambda_k^2}\right)\lambda^2_l\sec_M(e_k\wedge e_l)\gm(e_l,e_l)\\
&&+2\hspace{-8pt}\sum^m_{k=1,\,k\neq l}
\hspace{-9pt}f^*\gn(e_k,e_k)\Bigl\{\bigl(\lambda^2_l\sigma-\sigma_N(\df(e_k)\wedge\df(e_l))\bigr)
f^*\gn(e_l,e_l)\\
&&\hspace{4cm}+\lambda^2_l\sigma\bigl(\gm(e_l,e_l)-f^*\gn(e_l,e_l)\bigr)\Bigr\}\\
&&-2\lambda^2_l\sigma\sum_{k\neq l}f^*\gn(e_k,e_k)\gm(e_l,e_l)\\
&=&\sum^m_{k=1,\,k\neq l}\left(1+\frac{1-\lambda_k^2}{1+\lambda_k^2}\right)\lambda^2_l\sec_M(e_k\wedge e_l)\gm(e_l,e_l)\\
&&+2\sum^m_{k=1,\,k\neq l}
\lambda^2_lf^*\gn(e_k,e_k)\gm(e_l,e_l)\Bigl(\sigma-\sigma_N(\df(e_k)\wedge\df(e_l))\Bigr) \\
&&+\lambda^2_l\sigma\sum^m_{k=1,\,k\neq l}\left(\frac{1-\lambda_k^2}{1+\lambda_k^2}-1\right)\gm(e_l,e_l).
\end{eqnarray*}
We may then continue to get
\begin{eqnarray*}
\mathcal{B}(l)&=&2\sum^m_{k=1,\,k\neq l}
\lambda^2_lf^*\gn(e_k,e_k)\gm(e_l,e_l)\Bigl(\sigma-\sigma_N(\df(e_k)\wedge\df(e_l))\Bigr)\\
&&+\lambda^2_l\Bigl(\operatorname{Ric}_M(e_l,e_l)-(m-1)\sigma\,\gm(e_l,e_l)\Bigr)\\
&&+\sum^m_{k=1,\,k\neq l}\frac{\lambda^2_l}{1+\lambda^2_l}\frac{1-\lambda_k^2}{1+\lambda_k^2}
\Bigl(\sigma_M(e_k\wedge e_l)+\sigma\Bigr).
\end{eqnarray*}
Now we can see that
\begin{eqnarray*}
2\mathcal{B}&=&\sum^m_{l=1}\mathcal{B}(l) \\
&=&2\sum^m_{l,k=1,\,k\neq l}
\lambda^2_lf^*\gn(e_k,e_k)\gm(e_l,e_l)\Bigl(\sigma-\sigma_N(\df(e_k)\wedge\df(e_l))\Bigr)\\
&&+\sum^m_{l=1}\lambda^2_l\Bigl(\operatorname{Ric}_M(e_l,e_l)-(m-1)\sigma\,\gm(e_l,e_l)\Bigr)\\
&&+\sum_{1\le k<l \le m}\frac{(\lambda_k-\lambda_l)^2+2\lambda_l\lambda_k(1-\lambda_l\lambda_k)}
{(1+\lambda^2_k)(1+\lambda_l^2)}\Bigl(\sec_M(e_k\wedge e_l)+\sigma\Bigr).
\end{eqnarray*}
This completes the proof.
\end{proof}

\begin{lemma}\label{logu}
Under the assumptions of Theorem \ref{thmA}, the function $u$ satisfies the differential
inequality
$$\nabla_{\dt}\hspace{-2pt}\log u\ge\Delta\log u+\delta\|A\|^2,$$
for some positive real number $\delta.$
\end{lemma}

\begin{proof}
The proof follows by combining the estimates obtained in Lemma \ref{lem A} and \ref{lem B} with
the evolution equation of $\log u$ in Lemma \ref{lem evol}.
\end{proof}

\section{Proof of Theorem A}
Since the area decreasing property is preserved, the solution $F$ stays in particular
graphical for any $t\in[0,T)$. Recall that from Lemma \ref{logu}, we obtain the estimate
$$\nabla_{\dt}\hspace{-2pt}\log u\ge\Delta\log u+\delta\|A\|^2$$
for some positive number $\delta>0$. Once this estimate is available, we may proceed exactly
as in the paper by Wang and Tsui \cite{tsui} to exclude finite time singularities. Note, that
this step requires $N$ to be compact since both Nash's embedding theorem \cite{nash} and White's
regularity theorem \cite{white} for the mean curvature flow with controlled error terms (by the
compactness of $M\times N$) are applied to the mean curvature flow of
$$F(M)\subset M\times N\overset{\text{Nash}}{\subset}\real{p}.$$
Following the same arguments developed in the papers \cite[Section 4]{wang} or
\cite[Section 3]{lee}, we can prove the long-time existence and the convergence of
the mean curvature flow to a constant map.

% Literaturliste
%%%%%%%%%%%%%%%%%%%%%%%%%%%%%%%%%%%%%%%%%%%%%%%%%%%%%%%%%%%%%%%%%%%%%%%%%


\begin{bibdiv}
\begin{biblist}

\bib{chau-chen-he1}{article}{
   author={Chau, A.},
   author={Chen, J.},
   author={He, W.},
   title={Lagrangian mean curvature flow for entire Lipschitz graphs},
   journal={Calc. Var. Partial Differential Equations},
   volume={44},
   date={2012},
   number={1-2},
   pages={199--220},
%   issn={0944-2669},
%   review={\MR{2898776}},
%   doi={10.1007/s00526-011-0431-x},
}

\bib{chau-chen-he2}{article}{
   author={Chau, A.},
   author={Chen, J.},
   author={Yuan, Y.},
   title={Lagrangian mean curvature flow for entire Lipschitz graphs II},
   journal={Math. Ann.},
   volume={online first},
   date={2013},
%   number={1-2},
%   pages={199--220},
%   issn={0944-2669},
%   review={\MR{2898776}},
%   doi={10.1007/s00526-011-0431-x},
}

\bib{chen-li-tian}{article}{
   author={Chen, J.},
   author={Li, J.},
   author={Tian, G.},
   title={Two-dimensional graphs moving by mean curvature flow},
   journal={Acta Math. Sin. (Engl. Ser.)},
   volume={18},
   date={2002},
   number={2},
   pages={209--224},
%   issn={1439-8516},
%   review={\MR{1910957 (2003c:53094)}},
%   doi={10.1007/s101140200163},
}

\bib{eells}{article}{
   author={Eells, J.},
   author={Sampson, J.H.},
   title={Harmonic mappings of Riemannian manifolds},
   journal={Amer. J. Math.},
   volume={86},
   date={1964},
   pages={109--160},
   %issn={0002-9327},
   %review={\MR{0164306 (29 \#1603)}},
}

\bib{ecker}{article}{
   author={Ecker, K.},
   author={Huisken, G.},
   title={Mean curvature evolution of entire graphs},
   journal={Ann. of Math. (2)},
   volume={130},
   date={1989},
   number={3},
   pages={453--471},
   %issn={0003-486X},
   %review={\MR{1025164 (91c:53006)}},
   %doi={10.2307/1971452},
}

\bib{guth}{article}{
   author={Guth, L.},
   title={Homotopy non-trivial maps with small $k$-dilation},
   journal={arXiv:0709.1241v1},
   %volume={88},
   date={2007},
   pages={1--7},
   %issn={0003-486X},
   %review={\MR{0233295 (38 \#1617)}},
}

\bib{hamilton2}{article}{
   author={Hamilton, R.},
   title={Three-manifolds with positive Ricci curvature},
   journal={J. Differential Geom.},
   volume={17},
   date={1982},
   number={2},
   pages={255--306},
   %issn={0022-040X},
   %review={\MR{664497 (84a:53050)}},
}

\bib{lee}{article}{
   author={Lee, K.-W.},
   author={Lee, Y.-I.},
   title={Mean curvature flow of the graphs of maps between compact
   manifolds},
   journal={Trans. Amer. Math. Soc.},
   volume={363},
   date={2011},
   number={11},
   pages={5745--5759},
   %issn={0002-9947},
   %review={\MR{2817407}},
   %doi={10.1090/S0002-9947-2011-05204-9},
}

\bib{medos}{article}{
   author={Medo{\v{s}}, I.},
   author={Wang, M.-T.},
   title={Deforming symplectomorphisms of complex projective spaces by the
   mean curvature flow},
   journal={J. Differential Geom.},
   volume={87},
   date={2011},
   number={2},
   pages={309--341},
   %issn={0022-040X},
   %review={\MR{2788658 (2012h:53155)}},
}


\bib{nash}{article}{
   author={Nash, J.},
   title={The imbedding problem for Riemannian manifolds},
   journal={Ann. of Math. (2)},
   volume={63},
   date={1956},
   pages={20--63},
   %issn={0003-486X},
   %review={\MR{0075639 (17,782b)}},
}


\bib{savas}{article}{
   author={Savas-Halilaj, A.},
   author={Smoczyk, K.},
   title={The strong elliptic maximum principle for vector bundles and applications to minimal maps},
   journal={arXiv:1205.2379v1},
   %volume={88},
   date={2012},
   pages={1--39},
   %issn={0003-486X},
   %review={\MR{0233295 (38 \#1617)}},
}

\bib{smoczyk1}{article}{
   author={Smoczyk, K.},
   title={Mean curvature flow in higher codimension-Introduction and survey},
   journal={Global Differential Geometry,  Springer Proceedings in Mathematics},
   volume={12},
   date={2012},
   pages={231--274},
}

\bib{smoczyk2}{article}{
   author={Smoczyk, K.},
   title={Longtime existence of the Lagrangian mean curvature flow},
   journal={Calc. Var. Partial Differential Equations},
   volume={20},
   date={2004},
   number={1},
   pages={25--46},
%   issn={0944-2669},
%   review={\MR{2047144 (2004m:53119)}},
%   doi={10.1007/s00526-003-0226-9},
}


\bib{smoczyk3}{article}{
   author={Smoczyk, K.},
   title={Angle theorems for the Lagrangian mean curvature flow},
   journal={Math. Z.},
   volume={240},
   date={2002},
   number={4},
   pages={849--883},
   %issn={0025-5874},
   %review={\MR{1922733 (2003g:53120)}},
   %doi={10.1007/s002090100402},
}


\bib{smoczyk4}{article}{
   author={Smoczyk, K.},
   author={Wang, M.-T.},
   title={Mean curvature flows of Lagrangians submanifolds with convex
   potentials},
   journal={J. Differential Geom.},
   volume={62},
   date={2002},
   number={2},
   pages={243--257},
   %issn={0022-040X},
   %review={\MR{1988504 (2004d:53086)}},
}

\bib{tsui}{article}{
   author={Tsui, M.-P.},
   author={Wang, M.-T.},
   title={Mean curvature flows and isotopy of maps between spheres},
   journal={Comm. Pure Appl. Math.},
   volume={57},
   date={2004},
   number={8},
   pages={1110--1126},
   %issn={0010-3640},
   %review={\MR{2053760 (2005b:53110)}},
   %doi={10.1002/cpa.20022},
}

\bib{wang}{article}{
   author={Wang, M.-T.},
   title={Long-time existence and convergence of graphic mean curvature flow
   in arbitrary codimension},
   journal={Invent. Math.},
   volume={148},
   date={2002},
   number={3},
   pages={525--543},
   %issn={0020-9910},
   %review={\MR{1908059 (2003b:53073)}},
   %doi={10.1007/s002220100201},
}

\bib{wang1}{article}{
   author={Wang, M.-T.},
   title={Mean curvature flow of surfaces in Einstein four-manifolds},
   journal={J. Differential Geom.},
   volume={57},
   date={2001},
   number={2},
   pages={301--338},
   %issn={0022-040X},
   %review={\MR{1879229 (2003j:53108)}},
}

\bib{wang2}{article}{
   author={Wang, M.-T.},
   title={Deforming area preserving diffeomorphism of surfaces by mean
   curvature flow},
   journal={Math. Res. Lett.},
   volume={8},
   date={2001},
   number={5-6},
   pages={651--661},
   %issn={1073-2780},
   %review={\MR{1879809 (2003f:53122)}},
}


\bib{white}{article}{
   author={White, B.},
   title={A local regularity theorem for mean curvature flow},
   journal={Ann. of Math. (2)},
   volume={161},
   date={2005},
   number={3},
   pages={1487--1519},
   %issn={0003-486X},
   %review={\MR{2180405 (2006i:53100)}},
   %doi={10.4007/annals.2005.161.1487},
}


\end{biblist}
\end{bibdiv}
\end{document}